\documentclass[12pt]{amsart}
\usepackage[centertags]{amsmath}
\usepackage{amstext,amssymb,amsopn,amsthm}

\makeatletter
\@namedef{subjclassname@2010}{%
  \textup{2010} Mathematics Subject Classification}
\makeatother

\title
{Fractional Hardy--Sobolev--Maz'ya inequality on balls and halfspaces}
\author[B{.} Dyda]
{Bart{\l}omiej Dyda}
\date{April 28, 2010}
\address{
Faculty of Mathematics\\ University of Bielefeld\\
Postfach 10 01 31,
D-33501 Bielefeld, Germany
\and
 Institute of Mathematics and Computer Science\\ Wroc{\l}aw University of Technology\\
Wybrze\.ze Wyspia\'nskiego 27,
50-370 Wroc{\l}aw, Poland}
\email{bdyda (at) pwr wroc pl}

\subjclass[2010]{Primary 26D10; Secondary 46E35, 31C25}

\keywords{fractional Hardy-Sobolev-Maz'ya inequality,
  best constant, half-space, ball, regional fractional Laplacian, censored stable
  process}

\theoremstyle{plain}
\newtheorem{thm}{Theorem}

\newtheorem{lem}[thm]{Lemma}

\newtheorem{cor}[thm]{Corollary}

\theoremstyle{definition}

\theoremstyle{remark}
\newtheorem*{rem*}{Remark}

\newcommand{\R}{\mathbb{R}}

\DeclareMathOperator{\dist}{dist}

\begin{document}
\sloppy \footnotetext{
This research was partially supported by grant N N201 397137, MNiSW
and by the DFG through SFB-701 'Spectral Structures and Topological Methods in Mathematics'.
}

\begin{abstract}
We prove fractional Hardy--Sobolev--Maz'ya inequality for balls and a half-space,
partially answering the open problem posed by Frank and Seiringer \cite{FrankSeiringer}.
We note that for half-spaces this inequality has been recently obtained by Sloane \cite{Sloane}.
\end{abstract}
\maketitle
\section{Main result and discussion}\label{sec:i}
For an open set $D\subset\R^n$ and $0<\alpha<2$ let
\[
\mathcal{E}_D(u)=
\frac{1}{2}
\int_D \! \int_D
\frac{(u(x)-u(y))^2}{|x-y|^{n+\alpha}} \,dx\,dy, \quad u\in L^2(D).
\]
The quadratic form $\mathcal{E}$ is, up to a multiplicative constant, the
Dirichlet form of the censored stable process in $D$, see \cite{BBC}.
It has been shown \cite{LossSloane} that for convex, connected open sets $D$ and $1<\alpha<2$
\begin{equation}\label{Hardy}
\mathcal{E}_D(u) \geq \kappa_{n,\alpha} \int_D u^2(x)\, \delta_D^{-\alpha}(x)\,dx, \quad u\in C_c(D),
\end{equation}
where $\delta_D(x) = \dist(x,D^c)$ and
\begin{equation}
  \label{eq:ns}
\kappa_{n,\alpha}=
\pi^{\frac{n-1}{2}} \frac{\Gamma(\frac{1+\alpha}{2})}{\Gamma(\frac{n+\alpha}{2})}
\frac{B\left(\frac{1+\alpha}{2}, \frac{2-\alpha}{2}\right)
  -2^{\alpha}}{\alpha2^{\alpha}}
\end{equation}
is the \emph{largest} constant for which (\ref{Hardy}) holds.
In (\ref{eq:ns}) $B$ is the Euler beta function.

On the other hand, if $0<\alpha<n\wedge 2$, by Sobolev embedding \cite[(2.3)]{ChenKumagai}, 
 we have, e.g., for open convex, connected sets $D$
\begin{equation}\label{Sobolev}
\mathcal{E}_D(u) +\|u\|^2_{L^2} \geq c \left( \int_D |u(x)|^{2^*}\,dx \right)^{2/2^*}, \quad u\in C_c(D),
\end{equation}
with some $c=c(D,\alpha)>0$ and $2^*=2n/(n-\alpha)$.

Comparing (\ref{Hardy}) and (\ref{Sobolev}) an interesting question arises, whether
 the following Hardy-Sobolev-Maz'ya inequality
\begin{equation}\label{HardySobolevMazya}
\mathcal{E}_D(u) \geq \kappa_{n,\alpha} \int_D u^2(x)\, \delta_D^{-\alpha}(x)\,dx
+ c \left( \int_D |u(x)|^{2^*}\,dx \right)^{2/2^*}, \quad u\in C_c(D),
\end{equation}
 holds for $1<\alpha<n$ and convex domains $D$? A similar question was posed in \cite[page 2]{FrankSeiringer}.

The purpose of this note is to prove (\ref{HardySobolevMazya}) for a half-space and balls,
see Theorem~\ref{ballthm} and Corollary~\ref{halfspacecor}. We would like to note that while
writing this note a paper \cite{Sloane} of Sloane appeared, in which the author has proved
(\ref{HardySobolevMazya}) for half-spaces. However, our proof is different and we also obtain
(\ref{HardySobolevMazya}) for balls.

\section{Proofs}
We denote by $B_r=\{x\in \R^n:|x|<r\}$ the open Euclidean ball of radius $r>0$,
we set $B=B_1$ and by $S^{n-1}=\{x\in \R^n: |x|=1\}$ we denote the $(n-1)$-dimensional
unit sphere.

We define
\[
 L_D u(x) =  \lim_{\varepsilon\to 0^+}
\int_{D \cap \{|y-x|>\varepsilon\}}\frac{u(y)-u(x)}{|x-y|^{n+\alpha}} \,dy.
\]
Note that $L_D$ is, up to the multiplicative constant, the 
regional fractional Laplacian for an open set $D$, see
\cite{MR2214908}.  

Let
\begin{equation}\label{wn}
w_n(x)=(1-|x|^2)^{\frac{\alpha-1}{2}}, \quad x\in B\subset \R^n, n\geq 1.
\end{equation}
We recall from \cite[Lemma 2.1]{DydaHint} that
\[
 - L_{(-1,1)}w_1(x) = \frac{(1-x^2)^{\frac{-\alpha-1}{2}}}{\alpha}  \Bigg(
B\Big(\frac{\alpha+1}{2},\frac{2-\alpha}{2}\Big) - (1 - x)^\alpha + (1 + x)^\alpha \Bigg)
\]
hence by \cite[(2.3)]{DydaHint}
\begin{equation}\label{Lw1}
 - L_{(-1,1)}w_1(x) \geq c_1 (1-x^2)^{\frac{-\alpha-1}{2}} + c_2 (1-x^2)^{\frac{-\alpha+1}{2}},
\end{equation}
where
\[
 c_1=\frac{B \Big(\frac{\alpha+1}{2},\frac{2-\alpha}{2}\Big) - 2^\alpha}{\alpha},\quad
c_2=\frac{2^\alpha-2}{\alpha}.
\]

\begin{lem}\label{laplasjanupball}
We have for $w_n$ defined in (\ref{wn}) and $n\geq 2$
\[ 
 -L_Bw_n(x) \geq  \frac{c_1}{2} \int_{S^{n-1}}|h_n|^\alpha dh
\cdot  (1-|x|^2)^{-\frac{\alpha+1}{2}}
+ c_2|S^{n-1}| \cdot (1-|x|^2)^{-\frac{\alpha-1}{2}} 
\]
\end{lem}
\begin{proof}
Let $\mathbf{x}=(0,0,\ldots,0,x)$, $p=\frac{\alpha-1}{2}$.
We have
\begin{align*}
 -L_Bw_n(\mathbf{x}) &= p.v. \int_B \frac{(1-|\mathbf{x}|^2)^p - (1-|y|^2)^p}
     {|\mathbf{x}-y|^{n+\alpha}}\,dy\\
&=\frac{1}{2} \int_{S^{n-1}} dh
\; p.v.\int_{-xh_n  - \sqrt{x^2h_n^2-x^2+1}}^{-xh_n  + \sqrt{x^2h_n^2-x^2+1}}
\frac{ (1-|x|^2)^p - (1-|x+h t|^2)^p}{|t|^{1+\alpha}} \,dt.
\end{align*}
We calculate the inner principle value integral by changing the variable
$t=-xh_n + u \sqrt{x^2h_n^2-x^2+1}$
\begin{align*}
g(x,h)&:=
p.v.\int_{-xh_n  - \sqrt{x^2h_n^2-x^2+1}}^{-xh_n  + \sqrt{x^2h_n^2-x^2+1}}
\frac{ (1-|x|^2)^p - (1-|x+h t|^2)^p}{|t|^{1+\alpha}} \,dt\\
&=p.v. \int_{-1}^1
  \frac{(1-x^2)^p - (1-u^2)^p(1-x^2+x^2h_n^2)^p}
    {|-xh_n + u \sqrt{x^2h_n^2-x^2+1}|^{1+\alpha}}\,
 \sqrt{x^2h_n^2-x^2+1}\,du\\
&=
(1-x^2+x^2h_n^2)^{p-\alpha/2}
 p.v.\int_{-1}^1 \frac{ (1- \frac{x^2h_n^2}{1-x^2+x^2h_n^2})^p - (1-u^2)^p}
 { |u - \frac{xh_n}{\sqrt{1-x^2+x^2h_n^2}}|^{1+\alpha}}\,du\\
&=(1-x^2+x^2h_n^2)^{-1/2}
(-L_{(-1,1)}w_1)(\frac{xh_n}{\sqrt{1-x^2+x^2h_n^2}}).
\end{align*}
Hence by (\ref{Lw1}) we have
\begin{align*}
g(x,h) &\geq
(1-x^2+x^2h_n^2)^{-1/2} \bigg(
 c_1 (1- \frac{x^2h_n^2}{1-x^2+x^2h_n^2})^{\frac{\alpha-1}{2}-\alpha}\\
&\qquad\qquad\qquad\qquad\qquad\quad
  + c_2 (1- \frac{x^2h_n^2}{1-x^2+x^2h_n^2})^{\frac{\alpha-1}{2}-\alpha+1}
\bigg)\\
&= c_1(1-x^2+x^2h_n^2)^{\alpha/2}(1-x^2)^{-\frac{\alpha+1}{2}}+
 c_2(1-x^2+x^2h_n^2)^{\alpha/2-1}(1-x^2)^{-\frac{\alpha-1}{2}}\\
&\geq c_1 |h_n|^\alpha (1-x^2)^{-\frac{\alpha+1}{2}}
+ c_2 (1-x^2)^{-\frac{\alpha-1}{2}}.
\end{align*}
Thus
\begin{align*}
 -L_Bw_n(\mathbf{x}) &=
 \frac{1}{2} \int_{S^{n-1}}g(x,h) dh \\
&\geq
 \frac{c_1}{2} \int_{S^{n-1}}|h_n|^\alpha dh
\cdot  (1-x^2)^{-\frac{\alpha+1}{2}}
+ c_2|S^{n-1}| \cdot (1-x^2)^{-\frac{\alpha-1}{2}} 
\end{align*}
and we are done.
\end{proof}

\begin{cor}\label{Hball}
Let $1<\alpha<2$.
Let $w_n$ be as in (\ref{wn}) and let $n\geq 2$.
Then for every $u\in C_c(B)$,
\begin{align}
\mathcal{E}_B(u)&:=
\frac{1}{2}
\int_B \! \int_B
\frac{(u(x)-u(y))^2}{|x-y|^{n+\alpha}} \,dx\,dy\nonumber\\
&\geq
\frac{1}{2}
\int_B \! \int_B
\left(\frac{u(x)}{w(x)} - \frac{u(y)}{w(y)}\right)^2 
   \frac{w(x)w(y)}{|x-y|^{n+\alpha}} \,dx\,dy\nonumber\\
&+
{2^\alpha \kappa_{n,\alpha}}
\int_B {u^2(x)}
\;
(1-|x|^2)^{-\alpha}dx
+ 
c_3
\int_B {u^2(x)}
\;
(1-|x|^2)^{-\alpha+1}\,dx\label{hardyinball}
\end{align}
\end{cor}

\begin{proof}
The result follows from \cite[Lemma 2.2]{DydaHint} applied to $w=w_n$, Lemma~\ref{laplasjanupball}, 
and the following formula \cite[(7)]{LossSloane}
\[
 \kappa_{n,\alpha} = \kappa_{1,\alpha}\cdot \frac{1}{2}\int_{S^{n-1}} |h_n|^\alpha\,dh.
\]
\end{proof}

\begin{thm}\label{ballthm}
Let $1<\alpha<2$ and $n\geq 2$. There exist a constant $c=c(\alpha,n)$ such
that for every $0<r<\infty$ and $u\in C_c(B_r)$,
\begin{align}
\mathcal{E}_{B_r}(u)&:=
\frac{1}{2}
\int_{B_r} \! \int_{B_r}
\frac{(u(x)-u(y))^2}{|x-y|^{n+\alpha}} \,dx\,dy\nonumber\\
&\geq
{2^\alpha \kappa_{n,\alpha}}
\int_{B_r} {u^2(x)}
\;
r^\alpha (r^2-|x|^2)^{-\alpha}dx
+ 
c \left(\int_{B_r} {|u(x)|^{2^*}}\,dx \right) ^ {2/2^*}, \label{HSMball}
\end{align}
where $2^*=2n/(n-\alpha)$.
\end{thm}

\begin{proof}
By scaling, we may and do assume that $r=1$, that is we consider only the unit ball $B=B_1\subset \R^n$.
Recall (\ref{wn}) and
let $v=u/w_n$, $a_k=1-2^{-k}$ and $B_k=B(0,a_k)$. For $x,y\in B_{k_0}$ we have
\begin{align*}
 w_n(x)w_n(y) &= w_1(|x|)w_1(|y|) \geq
  \sum_{k=k_0}^\infty (w_1^2(a_{k}) - w_1^2(a_{k+1})),
\end{align*}
thus for any $x,y\in B$
\begin{align*}
 w_n(x)w_n(y) &\geq
 \sum_{k=1}^\infty\left(\Big( \frac{3}{2} \Big)^{\alpha-1}-1 \right) 2^{-k(\alpha-1)}
1_{B_k}(x) 1_{B_k}(y)\\
&\geq \sum_{k=1}^\infty \frac{\alpha-1}{2}\, 2^{-k(\alpha-1)}1_{B_k}(x) 1_{B_k}(y).
\end{align*}
It follows that
\begin{align}
\int_B \!& \int_B 
\left(\frac{u(x)}{w_n(x)} - \frac{u(y)}{w_n(y)}\right)^2 
   \frac{w_n(x)w_n(y)}{|x-y|^{n+\alpha}} \,dx\,dy + \int_B v^2(x)w^2(x)\,dx\nonumber\\
&\geq \frac{\alpha-1}{2}
\sum_{k=1}^\infty 2^{-k(\alpha-1)} \left( \int_{B_k} \! \int_{B_k}
   \frac{(v(x)-v(y))^2 }{|x-y|^{n+\alpha}} \,dx\,dy
+ \int_{B_k} v^2(x)\,dx \right).  \label{suma}
\end{align}
We write Sobolev inequality (\ref{Sobolev}) for $D=B_k$ and a function $v$
\begin{equation}\label{SobolevBk}
 \int_{B_k} \! \int_{B_k}
   \frac{(v(x)-v(y))^2 }{|x-y|^{n+\alpha}} \,dx\,dy
+ \int_{B_k} v^2(x)\,dx  \geq c \left( \int_{B_k} |v(x)|^{2^*}\,dx \right)^{2/2^*}.
\end{equation}
The constant $c=c(\alpha,n)$ in (\ref{SobolevBk}) may be chosen such that
it does not depend on $k$, because the radii $a_k$ of $B_k$ satisfy $1/2 \leq a_k \leq 1$.
By (\ref{suma}) and (\ref{SobolevBk}) we obtain
\begin{align*}
\int_B \!& \int_B 
\left(\frac{u(x)}{w_n(x)} - \frac{u(y)}{w_n(y)}\right)^2 
   \frac{w_n(x)w_n(y)}{|x-y|^{n+\alpha}} \,dx\,dy + \int_B v^2(x)\,dx\\
&\geq 
c \sum_{k=1}^\infty  2^{-k(\alpha-1)} \left( \int_{B_k} |v(x)|^{2^*}\,dx \right)^{2/2^*}\\
&\geq 
c \left( \sum_{k=1}^\infty \int_{B_k} 2^{\frac{-2^* k(\alpha-1)}{2}} |v(x)|^{2^*}\,dx \right)^{2/2^*}\\
&\geq
c' \left(  \int_{B} w_n(x)^{2^*} |v(x)|^{2^*}\,dx \right)^{2/2^*}
= c' \left( \int_B |u(x)|^{2^*}\,dx \right)^{2/2^*}.
\end{align*}
By this and Corollary~\ref{Hball} we obtain (\ref{HSMball}).
\end{proof}

\begin{cor}\label{halfspacecor}
Let $1<\alpha<2$, $n\geq 2$ and $\Pi=\R^{n-1}\times (0,\infty)$.
 There exist a constant $c=c(\alpha,n)$ such
that for every  $u\in C_c(B_r)$
\begin{align}
\mathcal{E}_{\Pi}(u)&:=
\frac{1}{2}
\int_{\Pi} \! \int_{\Pi}
\frac{(u(x)-u(y))^2}{|x-y|^{n+\alpha}} \,dx\,dy\nonumber\\
&\geq
{ \kappa_{n,\alpha}}
\int_{\Pi} {u^2(x)} x_n^{-\alpha}\,dx
+ 
c \left(\int_{\Pi} {|u(x)|^{2^*}}\,dx \right) ^ {2/2^*}, \label{HSMhalfspace}
\end{align}
where $2^*=2n/(n-\alpha)$.
\end{cor}

\begin{proof}
By Theorem~\ref{ballthm}
\[
\mathcal{E}_{B_r}(u) \geq
{\kappa_{n,\alpha}}
\int_{B_r} {u^2(x)} \delta_{B_r}(x)^{-\alpha}\,dx
+ 
c \left(\int_{B_r} {|u(x)|^{2^*}}\,dx \right) ^ {2/2^*},
\]
where $\delta_{B_r}(x) = \dist(x,B_r^c)$. Let $x_r=(0,\ldots,0,r)\in \Pi$,
by translation and inequality $\delta_{B(x_r,r)}(x) \leq x_n$ we obtain
\[
\mathcal{E}_{B(x_r,r)}(u) \geq
{\kappa_{n,\alpha}}
\int_{B(x_r,r)} {u^2(x)} x_n^{-\alpha}\,dx
+ 
c \left(\int_{B(x_r,r)} {|u(x)|^{2^*}}\,dx \right) ^ {2/2^*}.
\]
The corollary follows by letting $r\to\infty$.
\end{proof}


\begin{thebibliography}{1}

\bibitem{BBC}
K.~Bogdan, K.~Burdzy, and Z.-Q. Chen.
\newblock Censored stable processes.
\newblock {\em Probab. Theory Related Fields}, 127(1):89--152, 2003.

\bibitem{ChenKumagai}
Z.-Q. Chen and T.~Kumagai.
\newblock Heat kernel estimates for stable-like processes on {$d$}-sets.
\newblock {\em Stochastic Process. Appl.}, 108(1):27--62, 2003.

\bibitem{DydaHint}
B.~Dyda.
\newblock {F}ractional {H}ardy inequality with a remainder term.
\newblock http://www.math.uni-bielefeld.de/sfb701/preprints/sfb10014.pdf, 2010.

\bibitem{FrankSeiringer}
R.~Frank and R.~Seiringer.
\newblock {S}harp fractional {H}ardy inequalities in half-spaces.
\newblock arXiv:0906.1561v1 [math.FA], 2009.

\bibitem{MR2214908}
Q.-Y. Guan and Z.-M. Ma.
\newblock Reflected symmetric {$\alpha$}-stable processes and regional
  fractional {L}aplacian.
\newblock {\em Probab. Theory Related Fields}, 134(4):649--694, 2006.

\bibitem{LossSloane}
M.~Loss and C.~Sloane.
\newblock Hardy inequalities for fractional integrals on general domains.
\newblock arXiv:0907.3054v2 [math.AP], 2009.

\bibitem{Sloane}
C.~Sloane.
\newblock A {F}ractional {H}ardy-{S}obolev-{M}az'ya {I}nequality on the {U}pper
  {H}alfspace.
\newblock arXiv:1004.4828v1 [math.FA], 2010.

\end{thebibliography}

\def\cprime{$'$}

\end{document}